\documentclass{amsart}
\usepackage{amsmath,amssymb}
\usepackage[dvips]{graphics}
\usepackage[all]{xy}
\usepackage[usenames]{color}

\usepackage{pifont, amsfonts}
\usepackage{mathtools}
\usepackage[allowmove]{url}
\usepackage{multirow}
\usepackage{fancyhdr,ifthen}
\usepackage{enumerate}
\usepackage{amsmath}
\usepackage{amsfonts}
\usepackage{graphicx}

\newtheorem{theorem}{Theorem}[section]

\newtheorem{definition}[theorem]{Definition}
\newtheorem{lemma}[theorem]{Lemma}
\newtheorem{corollary}[theorem]{Corollary}

\allowdisplaybreaks

\numberwithin{equation}{section}

\begin{document}

\title{Topological Symmetry Groups of M\"{o}bius Ladders}
\date{\today}

\author{Erica Flapan and Emille Davie Lawrence}

 \subjclass{ 57M25, 57M15, 57M27, 92E10, 05C10}

\keywords{topological  symmetry groups, spatial graphs, molecular symmetries, M\"{o}bius ladders}

\address{Department of Mathematics, Pomona College, Claremont, CA 91711}
\email{eflapan@pomona.edu}
   \address{Department of Mathematics, University of San Francisco, San Francisco, CA 94117}
   \email{edlawrence@usfca.edu}

\thanks{The first author was supported in part by NSF grant DMS-0905087.}

\begin{abstract}  We classify all groups which can occur as the orientation preserving topological symmetry group of some embedding of a M\"{o}bius ladder graph in $S^3$.

\end{abstract}

\maketitle  

\pagenumbering{arabic} 

\section{Introduction}

In order to predict molecular behavior, it is useful to know the symmetries of a molecule. The group of rigid molecular symmetries, known as the {\it point group}, is a convenient way to represent this information.  This group is made up of the different rotations, reflections, and combinations of rotations and reflections which preserve a molecular structure in space.  However,  the larger a molecule is the more flexible it may be. In fact, macromolecules like DNA can be quite flexible.  Even some smaller molecules may be partially flexible or contain bonds around which a portion of the molecule can rotate.  Such a structure may have a symmetry that is induced by twisting or reflecting just part of the structure, while leaving the rest of the molecule fixed.    For example, on the left side of Figure~\ref{flex} we illustrate a molecular M\"{o}bius ladder which is flexible, and on the right side we illustrate a molecule that has rotating propellers on either end.  Not all of the symmetries of such molecules are included in the point group.  

\begin{figure}[htbp]
\begin{center}
\includegraphics[width=0.99\textwidth]{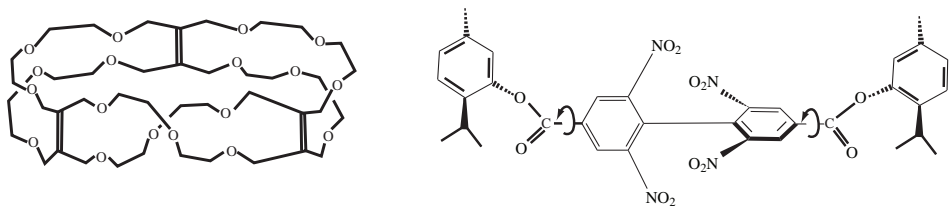}
\caption{The molecule on the left is flexible, while the one on the right has pieces which rotate.}
\label{flex}
\end{center}
\end{figure}

Thus it make sense to represent the symmetries of a non-rigid molecule by the group of automorphisms of the structure that can be induced by homeomorphisms rather than isometries of $\mathbb{R}^3$.  This group treats molecules as topological rather than geometric structures and hence we call this group of automorphisms the {\it topological symmetry group}.  Note that the topological symmetry group will always be finite because it is a subgroup of the automorphism group of a graph, even though the homeomorphisms that induce the automorphisms do not in general have finite order.

In this paper, we are interested in determining the topological symmetry groups of all embeddings of the family of M\"{o}bius ladders in $\mathbb{R}^3$. The molecular M\"{o}bius ladder, synthesized by Walba, Richards, and Haltiwanger ~\cite{WRH}, has the form of a three-runged ladder with its ends glued together with a half twist so that it resembles the boundary of a M\"{o}bius band (as illustrated in Figure~\ref{flex}). The rungs of the ladder are carbon-carbon double bounds, while the sides of the ladder consist of a long polyether chain.  This molecule is noteworthy because interest in it led to the development of the interdisciplinary field of topological stereochemistry.  In particular, in response to a question by the chemist David Walba who had synthesized the molecule, the topologist Jon Simon \cite{Si} proved that there is no orientation reversing homeomorphism of $\mathbb{R}^3$ taking the molecular M\"{o}bius ladder to itself in such a way a that rungs go to rungs and sides go to sides.  It followed from this result that even though the molecule is somewhat flexible, it must be chemically distinct from its mirror image.  Furthermore, Simon showed that no matter how many rungs are added to the molecular M\"{o}bius ladder, there would still be no orientation reversing homeomorphism of $\mathbb{R}^3$ taking the graph to itself.

While the motivation for defining the topological symmetry group came from studying non-rigid molecules, the study of symmetries of graphs in $\mathbb{R}^3$ is a natural extension of the study of symmetries of knots and links.  Furthermore, as is normally done in knot theory, we prefer to embed graphs in the 3-sphere $S^3=\mathbb{R}^3\cup \{\infty\}$ rather than in $\mathbb{R}^3$.  Note that an automorphism $f$ of a graph $\Gamma$ embedded in $\mathbb{R}^3$ is induced by a homeomorphism of $\mathbb{R}^3$ if and only if when $\Gamma$ is considered in $\mathbb{R}^3\subseteq S^3$ the automorphism $f$ is induced by a homeomorphism of $S^3$.  Thus the topological symmetry group of an embedded graph is unchanged if we consider the structure in $S^3$ rather than $\mathbb{R}^3$.  

In this paper, we consider embeddings of M\"{o}bius ladders that do not in general correspond to molecular structures.  As an abstract graph, a M\"{o}bius ladder with $n$ rungs, denoted by $M_{n}$, is a $2n$-gon with vertices at each corner and edges joining each antipodal pair of vertices.  There are infinitely many different ways to embed this graph in $S^3$, some containing knots, others containing twists between edges, and most looking nothing like the boundary of a M\"{o}bius band.  Flapan \cite{Fl} generalized Simon's result by proving that for every odd $n>3$ and any embedding $\Gamma$ of $M_n$ in $S^3$, there is no orientation reversing homeomorphism of $S^3$ taking $\Gamma$ to itself.  On the other hand, for every even $n$, there is an embedding of $\Gamma$ in $S^3$ which has an orientation reversing homeomorphism.

 We now consider the question of what groups can occur as topological symmetry groups of some embedding of $M_n$ in $S^3$.  This question has been addressed for complete graphs $K_n$ in a series of papers~\cite{CF, FMN, fmn2, fmny, fnt}.  Before we begin we introduce some terminology.  In particular, we will refer to a homeomorphism of $S^3$ taking an embedded graph $\Gamma$ to itself, as a homeomorphism of the pair $(S^3,\Gamma)$.  We also need the following definitions.
 
 \begin{definition}  Let $\gamma$ be an abstract graph.  The group of automorphisms of the vertices of $\gamma$ is denoted by $\mathrm{Aut}(\gamma)$.
 \end{definition}

\begin{definition}  Let $\Gamma$ be a graph embedded in $S^3$.  We define the {\bf orientation preserving topological symmetry group} $\mathrm{TSG}_+(\Gamma)$ as the subgroup of $\mathrm{Aut}(\Gamma)$ induced by orientation preserving homeomorphisms of $(S^3,\Gamma)$. 
\end{definition}

\begin{definition}  Let $G$ be a group and let $\gamma$ denote an abstract graph.  If there is some embedding $\Gamma$ of $\gamma$ in $S^3$ such that $\mathrm{TSG}_+(\Gamma)\cong G$, then we say that the group $G$ is {\bf positively realizable} for $\gamma$.
\end{definition}

The goal of this paper is to characterize for each $n$ which groups are positively realizable for $M_n$.  For $n=1$, the M\"{o}bius ladder $M_1$ is a $\theta$ graph which consists of two vertices and three edges going between them.  Because there are only two vertices, $\mathrm{Aut}(M_1)\cong \mathbb{Z}_2$.  A planar embedding $\Gamma$ of $M_1$, has $\mathrm{TSG}_+(\Gamma)\cong \mathbb{Z}_2$. An embedding $\Gamma$ of $M_1$ with a non-invertible knot in one of the edges of  $M_1$ has the trivial group as $\mathrm{TSG}_+(\Gamma)$.  This completely characterizes $\mathrm{TSG}_+(\Gamma)$ for embeddings $\Gamma$ of $M_1$.  For $n=2$, observe that $M_2=K_4$.  It was shown in \cite{CF} that every subgroup of $\mathrm{Aut}(K_4)\cong S_4$ is positively realizable for $K_4$.  Thus we focus on $M_n$ for $n\geq 3$.

For $n>3$, Simon \cite{Si} showed that every automorphism of the abstract graph $M_n$ takes the $2n$-gon to itself.  Thus for $n>3$, the group $\mathrm{Aut}(M_n)$ must be a subgroup of the dihedral group $D_{2n}$.  This makes the analysis of topological symmetry groups of embeddings of $M_n$ much simpler for $n>3$ than it is for $n=3$.  We consider the topological symmetry groups of embeddings of $M_3$ in Section 2, and consider the topological symmetry groups of embeddings of $M_n$ for all $n>3$ in Section 3.  Note that any embedding of $M_n$ with $n>1$ can be modified by adding distinct knots on each edge to obtain an embedding $\Gamma$ with the trivial group as $\mathrm{TSG}_+(\Gamma)$.  Thus from now on we will only consider non-trivial subgroups of  $\mathrm{Aut}(M_n)$.

\bigskip

\section{Topological symmetry groups of embeddings of $M_3$}

 In order to obtain a list of groups to check for positive realizability for $M_3$, we observe that $M_3$ has six vertices and hence its automorphism group is contained in the symmetric group $S_6$.  Thus we start with the following list of all subgroups of $S_6$ up to isomorphism \cite{Pf}:  $D_6$, $D_3$, $D_2$, $\mathbb Z_6$, $\mathbb Z_3$, $\mathbb Z_2$, $D_3\times D_3$, $\mathbb Z_3\times\mathbb Z_3$,  $(\mathbb Z_3\times\mathbb Z_3)\rtimes\mathbb Z_2$, $D_3\times\mathbb Z_3$, $\mathbb{Z}_2\times  \mathbb{Z}_2\times  \mathbb{Z}_2$, $S_3\wr\mathbb Z_2$, $\mathbb Z_4$, $\mathbb Z_5$, $D_4$, $\mathbb Z_2\times\mathbb Z_4$, $D_4$, $D_5$, $A_4$, $D_4\times\mathbb Z_2$, $\mathbb Z_5\rtimes\mathbb Z_4$, $A_4\times\mathbb Z_2$, $S_4$, $(\mathbb Z_3\times\mathbb Z_3)\rtimes Z_4$, $S_4\times \mathbb Z_2$, $A_5$, $S_5$, $A_6$, $S_6$.

We can see from Figure~\ref{m3} that the graph $M_3$ is equivalent to the complete bipartite graph $K_{3,3}$.  Also not that the automorphism group of $K_{3,3}$ is the same as that of its complementary graph, two triangles.  Hence $\mathrm{Aut}(K_{3,3})$ is isomorphic to the wreath product $S_3\wr\mathbb Z_2$. However, a result of Nikkuni and Taniyama \cite{NT} shows that, up to conjugation, the only non-trivial automorphisms of $K_{3,3}$ that can be induced by orientation preserving homeomorphisms of some embedding of $K_{3,3}$ in $S^3$ are: $(123)$, $(12)(45)$, $(123)(456)$, $(14)(25)(36)$, and $(142536)$.  In particular, every such automorphism has order 2, 3, or 6, and no such automorphism is a transposition.  It follows that for any embedding $\Gamma$ of $K_{3,3}$ in $S^3$, the group $\mathrm{TSG}_+(\Gamma)$ is a proper subgroup of $S_3\wr\mathbb Z_2$.  Since the order of $S_3\wr\mathbb Z_2$ is 72, this means the order of any positively realizable group for $K_{3,3}$ must be a proper divisor of 72.

\begin{figure}[h]
\begin{center}
\includegraphics[scale=.6]{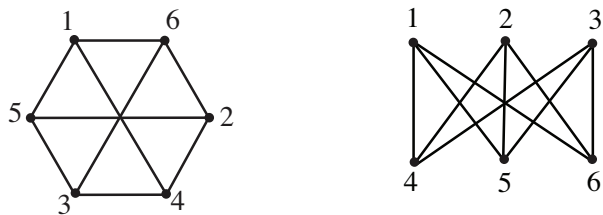}
\caption{$K_{3,3}$ and $M_3$ are the same graph.}
\label{m3}
\end{center}
\end{figure}

Furthermore, since no automorphisms of order 4 or 5 of $K_{3,3}$ can be induced by an orientation preserving homeomorphism of some embedding of $K_{3,3}$ in $S^3$, we only need to consider subgroups of $S_6$ containing no elements of order 4 or 5.  Here is a complete list of the subgroups we need to consider up to isomorphism: $D_6$, $D_3$, $D_2$, $\mathbb Z_6$, $\mathbb Z_3$, $\mathbb Z_2$, $D_3\times D_3$, $\mathbb Z_3\times\mathbb Z_3$,  $(\mathbb Z_3\times\mathbb Z_3)\rtimes\mathbb Z_2$, $D_3\times\mathbb Z_3$, $\mathbb{Z}_2\times  \mathbb{Z}_2\times  \mathbb{Z}_2$.

  Since no transposition of $K_{3,3}$ can be induced by an orientation preserving homeomorphism of $(S^3, \Gamma)$ for any embedding $\Gamma$ of $K_{3,3}$ in $S^3$, the following lemma implies that the group $\mathbb{Z}_2\times  \mathbb{Z}_2\times  \mathbb{Z}_2$ is not positively realizable for $K_{3,3}$.

\begin{lemma}\label{Z2Z2Z2}  Let $K_{3,3}$ be the complete bipartite graph with vertex sets $\{1,2,3\}$ and $\{4,5,6\}$.  Let $G\leq \mathrm{Aut}(K_{3,3})$ such that $G\cong \mathbb{Z}_2\times  \mathbb{Z}_2\times  \mathbb{Z}_2$.  Then $G$ contains a transposition.
\end{lemma}

\begin{proof}  Suppose that $G$ does not contain a transposition.  Since $G\cong \mathbb{Z}_2\times  \mathbb{Z}_2\times  \mathbb{Z}_2$, every nontrivial element of $G$ has order 2.  Hence every element of $G$ either interchanges the sets $\{1,2,3\}$ and $\{4,5,6\}$ or interchanges pairs of vertices in each of the sets $\{1,2,3\}$ and $\{4,5,6\}$.  The product of two distinct elements of $G$ that interchange the sets $\{1,2,3\}$ and $\{4,5,6\}$ is an element of $G$ which interchanges one pair of vertices in each of the sets $\{1,2,3\}$ and $\{4,5,6\}$.  Since $\mathbb{Z}_2\times  \mathbb{Z}_2\times  \mathbb{Z}_2$ contains eight involutions, it must contain at least  two elements that interchange a pair of vertices in each of the sets $\{1,2,3\}$ and $\{4,5,6\}$. Thus without loss of generality $G$ contains $(12)(45)$, and either $(12)(46)$ or $(13)(46)$.  But this is impossible because then $G$ would contain either $(12)(45)(12)(46)=(456)$ or $(12)(45)(13)(46)=(123)(456)$, both of which have order 3.  Thus $G$ must contain a transposition. \end{proof}
\medskip

We will now show that all of the groups $D_6$, $D_3$, $D_2$, $\mathbb Z_6$, $\mathbb Z_3$, $\mathbb Z_2$, $D_3\times D_3$, $\mathbb Z_3\times\mathbb Z_3$,  $(\mathbb Z_3\times\mathbb Z_3)\rtimes\mathbb Z_2$, $D_3\times\mathbb Z_3$ are positively realizable for $K_{3,3}$.  Observe that the group $D_6$ contains the groups $D_3$, $D_2$, $\mathbb Z_6$, $\mathbb Z_3$, $\mathbb Z_2$,  as subgroups, and the group $D_3\times D_3$ contains the groups $\mathbb Z_3\times\mathbb Z_3$,  $(\mathbb Z_3\times\mathbb Z_3)\rtimes\mathbb Z_2$, and $D_3\times\mathbb Z_3$ as subgroups.  We deal with these two sets of groups in different subsections.  As a matter of convenience, we will refer to the following automorphisms of $K_{3,3}$.

\begin{eqnarray*}
f&=&(123)(456)\\
g&=&(123)(465)\\
\psi&=&(14)(25)(36)\\
\varphi&=&(12)(45)
\end{eqnarray*}

 It is easily checked that the above automorphisms have the following relations:  $fg=gf$, $f\psi=\psi f$, $\psi g\psi=g^{-1}$, $\varphi f\varphi=f^{-1}$, $\varphi g\varphi=g^{-1}$, and $\phi(f\psi)\varphi=(f\psi)^{-1}$  
\medskip

\subsection{$D_6$ and its subgroups.}

In this subsection we create embeddings of $K_{3,3}$ in $S^3$ to show that the dihedral group $D_6$ and its subgroups  $D_3$, $D_2$, $\mathbb Z_6$, $\mathbb Z_3$, and $\mathbb Z_2$ are positively realizable for $K_{3,3}$.  In each case, we will start with a planar embedding of a hexagon with consecutive labeled vertices 1, 6, 2, 4, 3, 5 and stacked edges between antipodal vertices.  Then we add knots in various edges to obtain the groups we want.  In particular, we use knots to ensure that any homeomorphism of the graph in $S^3$ will take the hexagon to itself.  This will guarantee that the topological symmetry group of each embedding is a subgroup of $D_6$.

Note that it was shown in \cite{FMN} that for 3-connected graphs adding local knots to edges is a well defined operation and that any homeomorphism of $S^3$ taking the graph to itself must take an edge with a given knot to an edge with the same knot.  Furthermore, any orientation preserving homeomorphism of $S^3$ taking the graph to itself must take an edge with a given non-invertible knot to an edge with the same knot oriented in the same way.  

\begin{theorem}\label{D6}
Every subgroup of $D_6$ is positively realizable for $K_{3,3}$.
\end{theorem}

\begin{proof}  We have already shown that the trivial group is realizable, so we only need to consider $D_6$ and its subgroups  $D_3$, $D_2$, $\mathbb Z_6$, $\mathbb Z_3$, and $\mathbb Z_2$.  

We begin with the embedding $\Gamma$ of $K_{3,3}$ illustrated in Figure~\ref{D6}, in which all of the black squares represent the same invertible knot.  Because of these knots any homeomorphism of $(S^3, \Gamma)$ takes the hexagon to itself, and hence $\mathrm{TSG}_+(\Gamma)$ must be isomorphic to subgroup of $D_6$. We define a homeomorphism of $(S^3, \Gamma)$ by composing a counterclockwise rotation of the hexagon by $2\pi/6$ and a $2\pi/3$ meridional rotation around the hexagon, then isotoping the knots back into position.   This homeomorphism induces the automorphism $f\psi=(153426)$.   Also a rotation of $S^3$ pointwise fixing an axis containing the rung $\overline{36}$ takes $\Gamma$ to itself inducing the automorphism $\varphi=(12)(45)$.  Since $\langle f\psi,\varphi\rangle\cong D_6$, we know that $\mathrm{TSG}_+(\Gamma)\cong D_6$.  

\begin{figure}[h]
\begin{center}
\includegraphics[scale=.6]{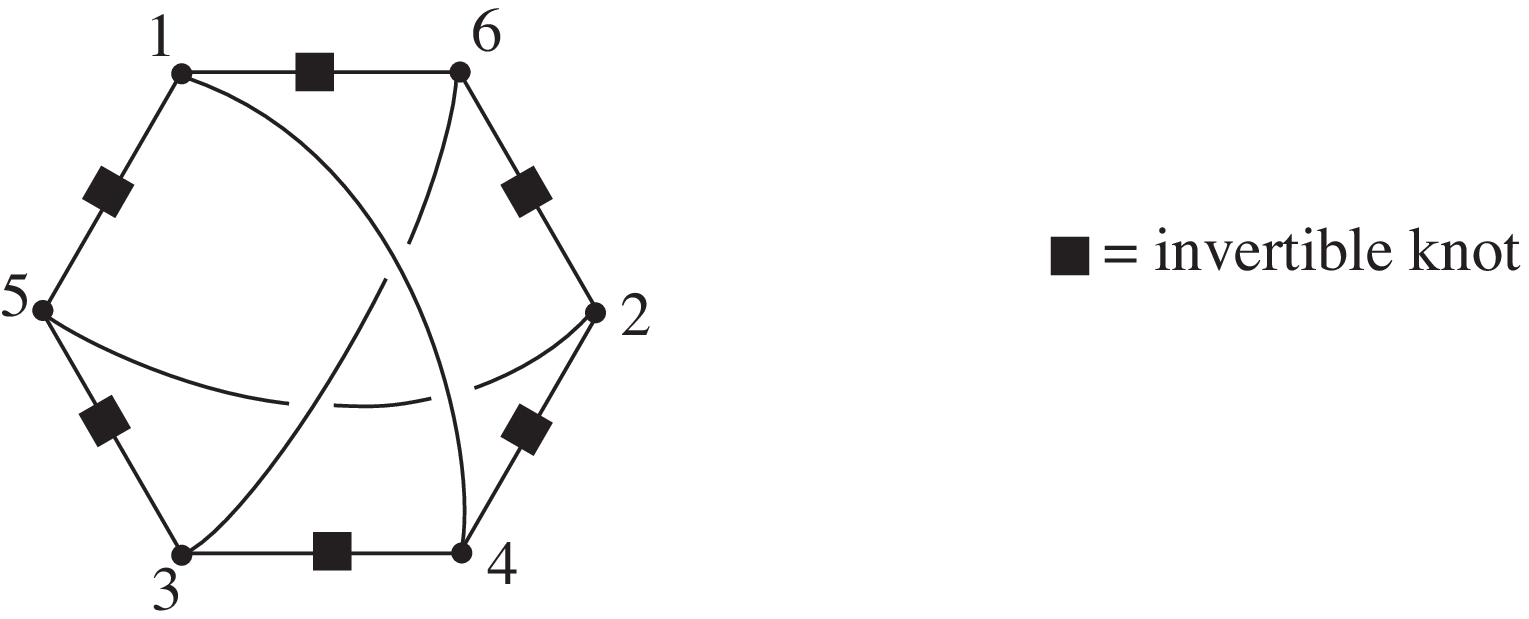}
\caption{$\mathrm{TSG_+}(\Gamma)\cong D_6$.}
\label{D6}
\end{center}
\end{figure}

Starting with the embedding $\Gamma$ in Figure~\ref{D6}, we obtain a new embedding $\Gamma_1$ of $K_{3,3}$ by replacing each invertible knot on $\Gamma$ with equivalent non-invertible knots. The automorphism $f\psi=(153426)$ is induced on $\Gamma_1$ by an analogous homeomorphism of $(S^3, \Gamma_1)$.  However, because of the non-invertible knots there is no orientation preserving homeomorphism of $(S^3,\Gamma_1)$ which turns the hexagon over.  Hence $\mathrm{TSG}_+(\Gamma_1)=\langle f\psi\rangle\cong\mathbb Z_6$.

\begin{figure}[h]
\begin{center}
\includegraphics[scale=.6]{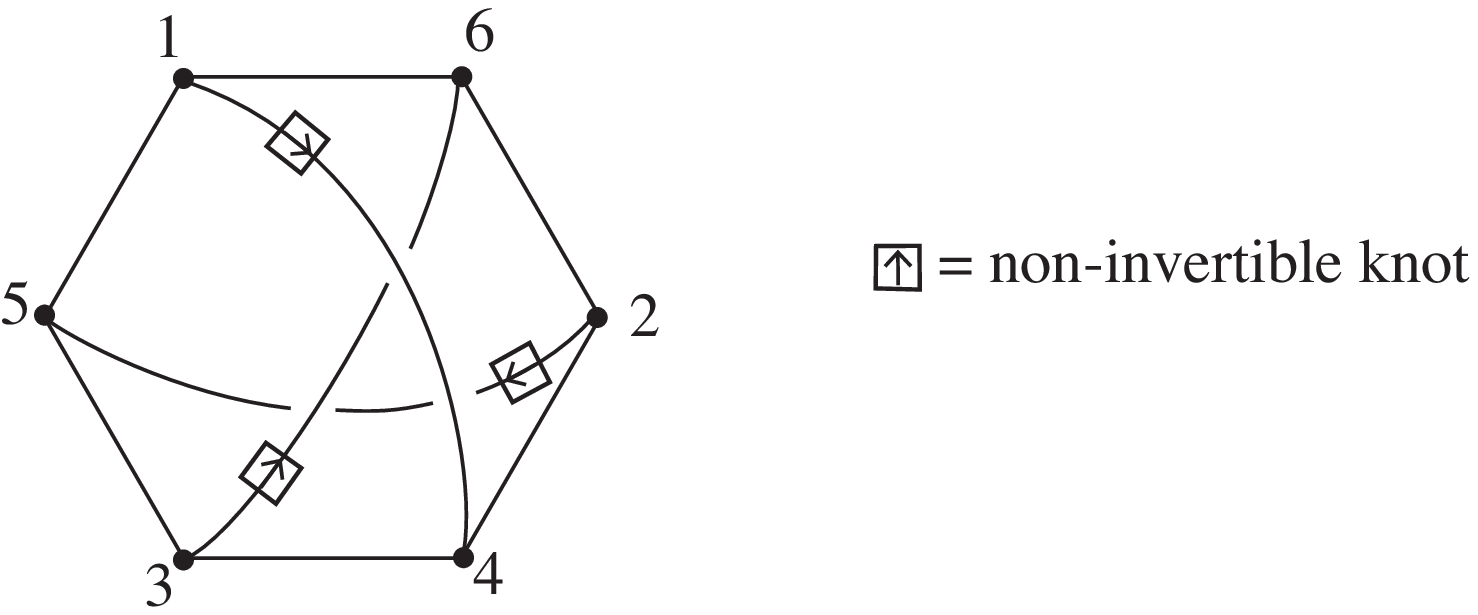}
\caption{$\mathrm{TSG_+}(\Gamma_2)\cong D_3$}
\label{D3}
\end{center}
\end{figure}

Next let $\Gamma_2$ be the embedding of $K_{3,3}$ shown in Figure~\ref{D3}.  The boxed arrows represent a particular non-invertible knot which is oriented as indicated.  Because of these non-invertible knots the hexagon must be invariant under any homeomorphism of  $(S^3,\Gamma_2)$, and no even order automorphism of the hexagon can be induced by an orientation preserving homeomorphism.  Thus $\mathrm{TSG}_+(\Gamma_2)$ must be isomorphic to subgroup of $D_3$.  Now the automorphism $\varphi$ is induced by a rotation of $(S^3,\Gamma_2)$ pointwise fixing an axis containing the edge $\overline{36}$. We define a homeomorphism $h_2$ of $(S^3, \Gamma_2)$ by composing a clockwise $2\pi/3$ rotation and a $2\pi/3$ rotation meridionally around the hexagon.  Then $h_2$ induces the automorphism $f=(123)(456)$ of $\Gamma_2$.  Therefore, $\mathrm{TSG}_+(\Gamma_2)=\langle f,\varphi\rangle\cong D_3$. 

Starting with the embedding $\Gamma_2$ illustrated in Figure~\ref{D3}, we nowF add noninvertible knots to every other edge of the hexagon to obtain a new embedding $\Gamma_3$ of $K_{3,3}$ as shown in Figure~\ref{Z3}.  Because the knots on the hexagon are noninvertible,  no orientation preserving homeomorphism of $(S^3, \Gamma_3)$ turns the hexagon over.  But we can still induce the automorphism $f=(123)(456)$.  Thus, in this case, $\mathrm{TSG}_+(\Gamma_3)=\langle f\rangle\cong \mathbb Z_3$.

\begin{figure}[h]
\begin{center}
\includegraphics[scale=.6]{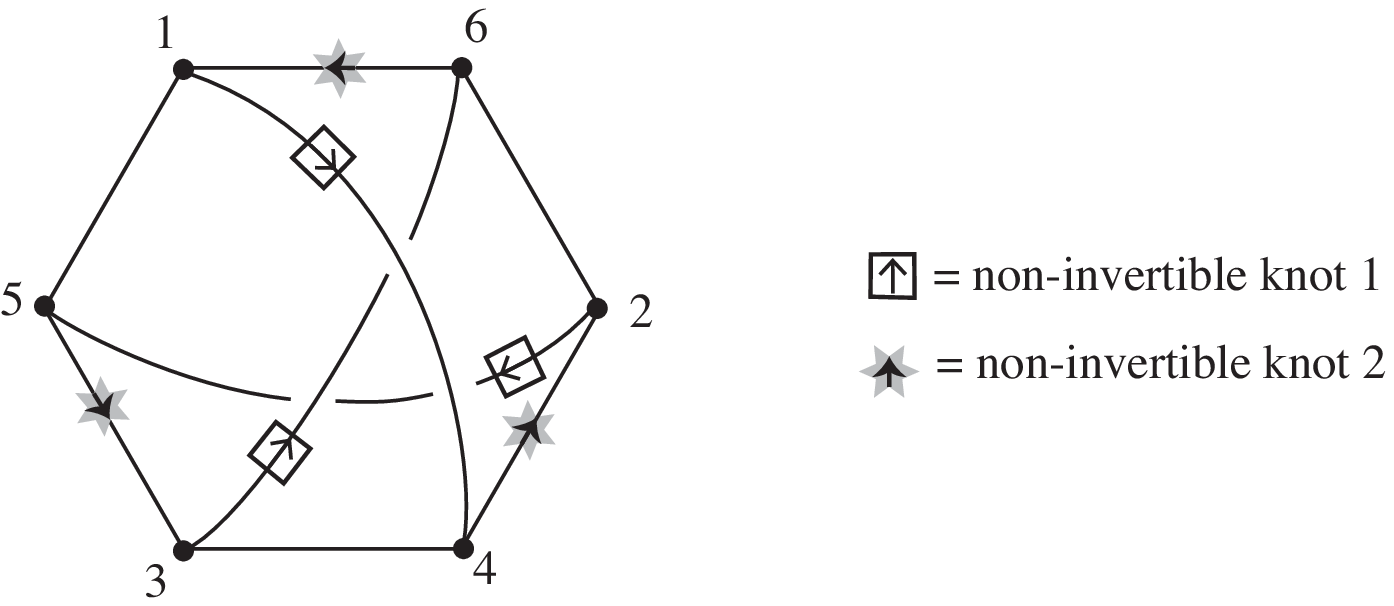}
\caption{$\mathrm{TSG_+}(\Gamma_3)\cong \mathbb Z_3$}
\label{Z3}
\end{center}
\end{figure}

Now, consider the embedding $\Gamma_4$ of $K_{3,3}$ in Figure~\ref{D2} which has equivalent invertible knots in edges $\overline{15}$ and $\overline{24}$ and a different invertible knot in the remaining edges of the hexagon.

\begin{figure}[h]
\begin{center}
\includegraphics[scale=.6]{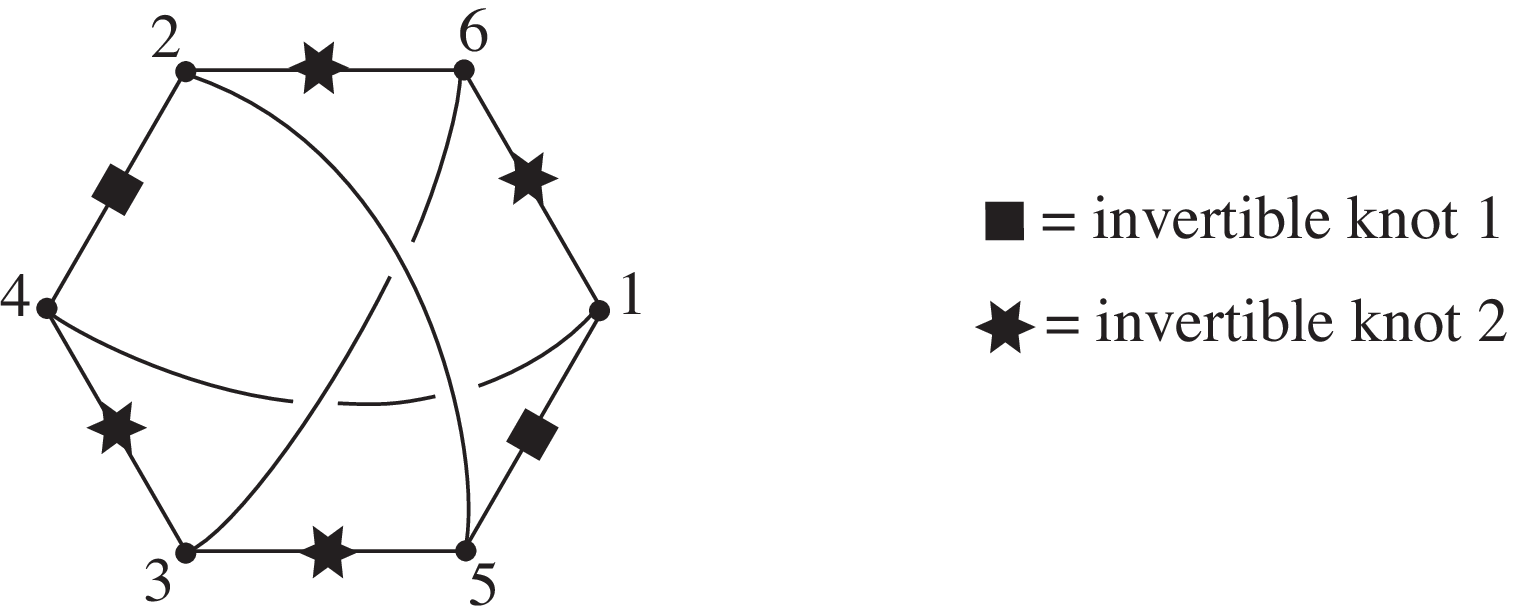}
\caption{$\mathrm{TSG_+}(\Gamma_4)\cong D_2$}
\label{D2}
\end{center}
\end{figure}

Any homeomorphism of $(S^3, \Gamma_4)$ must take the hexagon to itself either setwise fixing both $\overline{15}$ and $\overline{24}$ or interchanging them.  Rotating the hexagon by $180^{\circ}$ induces the automorphism $\psi=(14)(25)(36)$, and turning over the graph about an axis containing the edge $\overline{36}$ induces $\varphi$.  The product $\psi\varphi=(15)(24)(36)$ is the only nontrivial automorphism which setwise fixes both $\overline{15}$ and $\overline{24}$. Thus, $\mathrm{TSG}_+(\Gamma_4)=\langle \psi,\varphi\rangle\cong D_2$.

Finally, consider the embedding $\Gamma_5$ of $K_{3,3}$ shown in Figure~\ref{Z2} with four distinct invertible knots on the edges of the hexagon as indicated. The homeomorphism of $S^3$ which turns the graph over about an axis through edges $\overline{14}$ and $\overline{25}$ induces the automorphism $\psi=(14)(25)(36)$ on $\Gamma_5$. Since $\overline{14}$ and $\overline{25}$ must each be setwise fixed by any homeomorphism of $(S^3,\Gamma_5)$, the automorphism $\psi$ is the only nontrivial element in $\mathrm{TSG_+}(\Gamma_5)$.  Thus, $\mathrm{TSG_+}(\Gamma_5)=\langle\psi\rangle\cong \mathbb Z_2$.\end{proof}

\begin{figure}[h]
\begin{center}
\includegraphics[scale=.6]{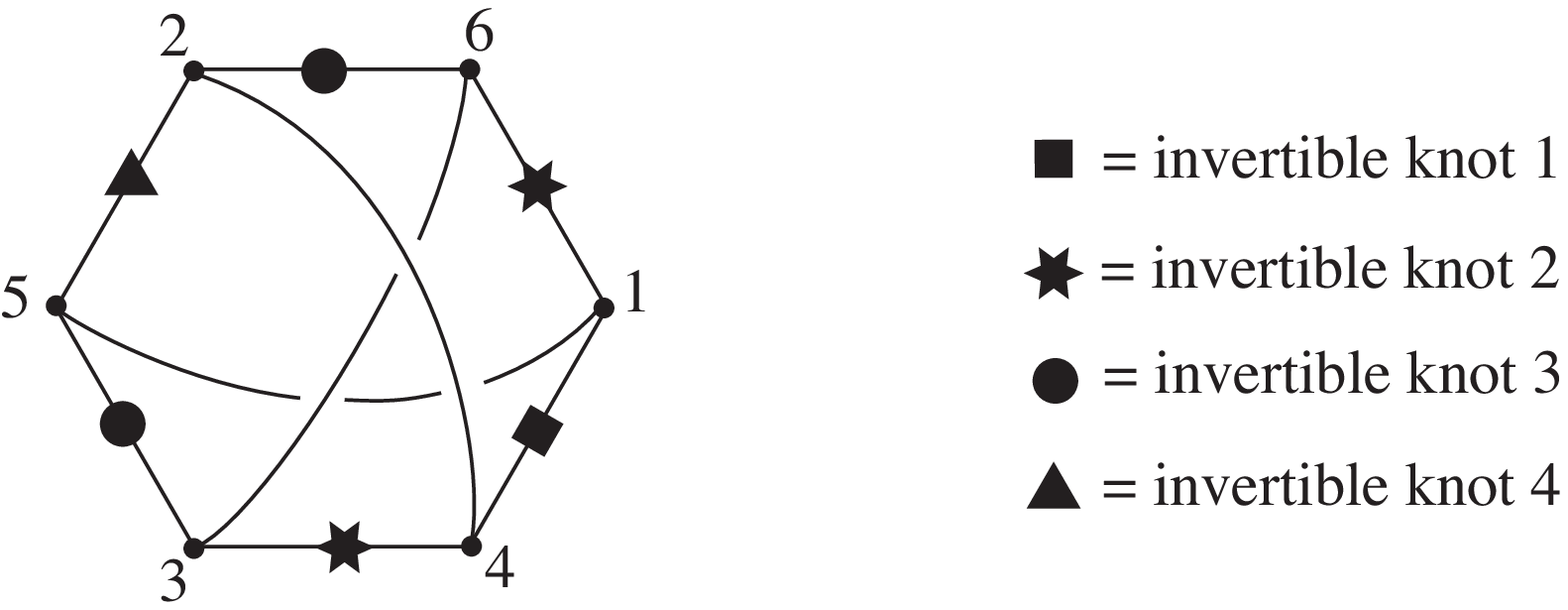}
\caption{$\mathrm{TSG_+}(\Gamma_5)\cong Z_2$}
\label{Z2}
\end{center}
\end{figure}
\medskip

\subsection{$D_3\times D_3$ and its subgroups.}
In this subsection we create embeddings of $K_{3,3}$ in $S^3$ to show that the group $D_3\times D_3$ and its subgroups $(\mathbb Z_3\times\mathbb Z_3)\rtimes\mathbb Z_2$, $D_3\times\mathbb Z_3$, and $\mathbb{Z}_3\times \mathbb{Z}_3$ are positively realizable for $K_{3,3}$.  In particular, we prove the following theorem.

\begin{theorem}\label{D3D3}
Every subgroup of $D_3\times D_3$ is positively realizable for $K_{3,3}$
\end{theorem}

\begin{proof}  Since we proved that all of the subgroups of $D_6$ are positively realizable in Theorem~\ref{D6}, here we only need to prove that the group $D_3\times D_3$ and its subgroups $(\mathbb Z_3\times\mathbb Z_3)\rtimes\mathbb Z_2$, $D_3\times\mathbb Z_3$, and $\mathbb{Z}_3\times \mathbb{Z}_3$ are positively realizable for $K_{3,3}$

\begin{figure}[h]
\begin{center}
\includegraphics[scale=.85]{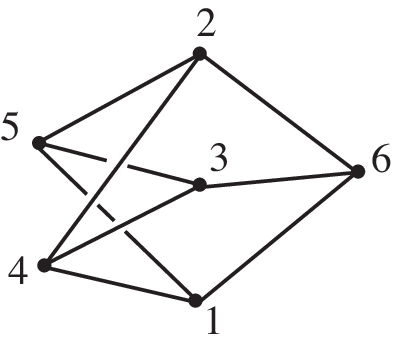}
\caption{$\mathrm{TSG}_+(\Gamma)\cong D_3\times D_3$.}
\label{D3D3}
\end{center}
\end{figure}

We begin with an embedding $\Gamma$ of $K_{3,3}$ in the form of a fan with three blades as shown in Figure~\ref{D3D3}.  We can think of $S^3$ as the union of two solid tori with the core of one solid torus containing the vertices $1$, $2$, and $3$, and the core of the other solid torus containing vertices $4$, $5$, and $6$.  The automorphisms $f$ and $g$ are induced on $\Gamma$ by order $3$ rotations about the cores of both solid tori. Turning over the graph about an axis containing the edge $\overline{36}$ induces the automorphism $\varphi$ on $\Gamma$.   Finally, the homeomorphism of $S^3$ which interchanges the two cores induces $\psi=(14)(25)(36)$.  Therefore, $\langle f,\varphi,g,\psi\rangle\cong D_3\times D_3\leq \mathrm{TSG}_+(\Gamma)$.  However, since the order of any positively realizable group for $K_{3,3}$ must be a proper divisor of 72 and the order of $D_3\times D_3$ is 36, we know that in fact $\mathrm{TSG}_+(\Gamma)\cong D_3\times D_3$.

Now, we modify the fan embedding $\Gamma$ illustrated in Figure~\ref{D3D3} by placing a non-invertible knot in each edge, giving us a new embedding $\Gamma_1$ of $K_{3,3}$ shown in Figure~\ref{Z3Z3semidirectZ2}.  Observe that the non-invertible knots are oriented from vertices $4$, $5$, and $6$ towards vertices $1$, $2$, and $3$.  Hence no orientation preserving homeomorphism of $(S^3,\Gamma_1)$ can interchange the set of vertices $\{1,2,3\}$ with the set of vertices $\{4,5,6\}$.  Also, we know that there is no transposition in $\mathrm{TSG}_+(\Gamma_1)$.  On the other hand, the automorphisms $f$, $g$, and $\varphi$ can be induced by homeomorphisms of $(S^3,\Gamma_1)$ as they were for $(S^3,\Gamma)$.  Recall that we have the relations $\varphi f\varphi=f^{-1}$ and $\varphi g\varphi=g^{-1}$.  Thus $\mathrm{TSG}_+(\Gamma_1)=\langle f,g,\varphi\rangle\cong (\mathbb Z_3\times\mathbb Z_3)\rtimes \mathbb Z_2$.

\begin{figure}[htbp]
\begin{center}
\includegraphics[scale=.85]{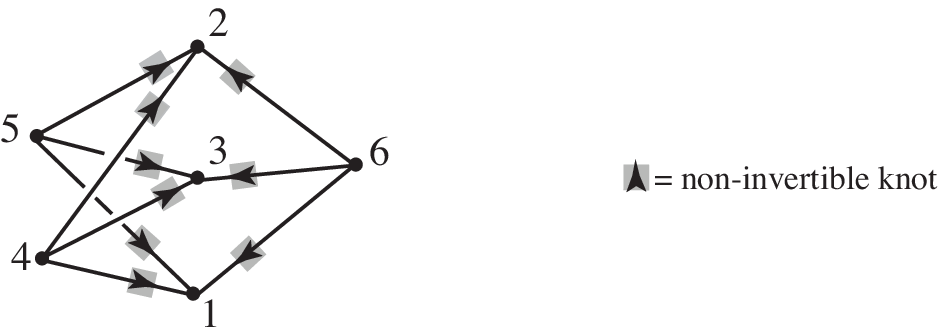}
\caption{$\mathrm{TSG}_+(\Gamma_1)\cong (\mathbb Z_3\times Z_3)\rtimes \mathbb Z_2$.}
\label{Z3Z3semidirectZ2}
\end{center}
\end{figure}

In order to find an embedding $\Gamma_2$ of $K_{3,3}$ such that $\mathrm{TSG}_+(\Gamma_2)\cong D_3\times \mathbb{Z}_3$, we again start with the embedding $\Gamma$ illustrated in Figure~\ref{D3D3}.  Now in a small neighborhood of each vertex $x\in \{1,2,3\}$ and $a\in \{4,5,6\}$ we embed the arcs of the edges as illustrated in Figure~\ref{knotted}.  This gives us an embedding $\Gamma_2$ of $K_{3,3}$.

\begin{figure}[here]
\begin{center}
\includegraphics[width=0.6\textwidth]{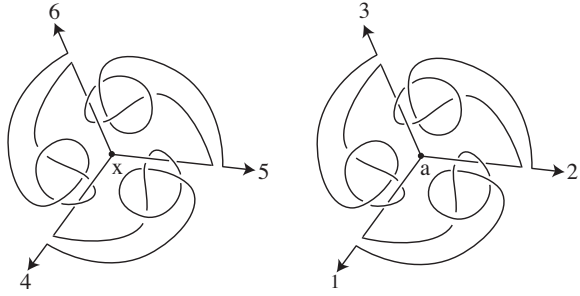}
\caption{These knots orient the three edges around each vertex.}
\label{knotted}
\end{center}
\end{figure}

If a local knot in an edge $e_1$ is linked with an edge $e_2$ (as $\overline{6x}$ is with $\overline{4x}$ in Figure~\ref{knotted}) we will say that $e_1$ is {\it knotted around} $e_2$.  Thus for each vertex $x\in \{1,2,3\}$, the edge $\overline {x4}$ is knotted around the edge $\overline{x5}$, which in turn is knotted around the edge $\overline{x6}$, which in turn is knotted around the edge $\overline{x4}$.  Similarly, for each vertex $a\in \{4,5,6\}$, the edge $\overline {a1}$ is knotted around the edge $\overline{a2}$, which in turn is knotted around the edge $\overline{a3}$, which in turn is knotted around the edge $\overline{a1}$.  On the other hand, none of the reverse knotting around relations hold. Suppose that there were some homeomorphism of $(S^3, \Gamma_2)$ that interchanged pairs of vertices in each of the sets $\{1,2,3\}$ and $\{4,5,6\}$.  Without loss of generality, suppose that some homeomorphism of $(S^3, \Gamma_2)$ induces the automorphism $\varphi=(12)(45)$.  Thus the edges $\overline{14}$ and $\overline{15}$, are mapped to the edges $\overline{25}$ and $\overline{24}$ respectively.  But this is impossible because the edge $\overline{14}$ is knotted around the edge $\overline{15}$, while the edge $\overline{25}$ is not knotted around the edge $\overline{24}$.  Thus no homeomorphism of $(S^3, \Gamma_2)$ interchanges pairs of vertices within the sets $\{1,2,3\}$ and $\{4,5,6\}$. 

The automorphisms $f$ and $g$ are induced on $\Gamma_2$ by order $3$ rotations about the cores of both solid tori as they were induced on $\Gamma$.  Figure~\ref{psiD3Z3} illustrates an isotopy of $S^3$ inducing the automorphism $\psi=(14)(25)(36)$ on the embedding $\Gamma$ while preserving the orientation of the edges around each vertex.  A similar isotopy induces $\psi$ on $\Gamma_2$ preserving the knots around the edges.  Thus $\psi \in \mathrm{TSG}_+(\Gamma_2)$.   It follows that $\mathrm{TSG}_+(\Gamma_2)=\langle f,g,\psi\rangle\cong D_3\times \mathbb{Z}_3$.

\begin{figure}[here]
\begin{center}
\includegraphics[scale=.8]{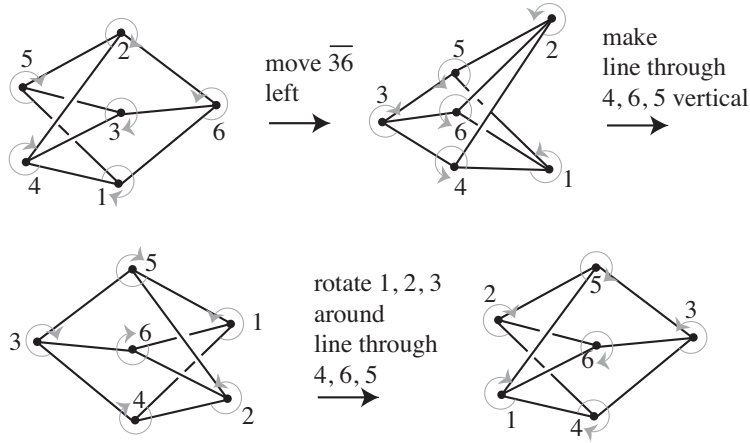}
\caption{This isotopy induces $\psi$ on $\Gamma$ preserving the orientation of the edges around each vertex.}
\label{psiD3Z3}
\end{center}
\end{figure}

Finally, consider the embedding $\Gamma_3$  obtained from the embedding $\Gamma_2$ by adding a non-invertible knot  to each edge oriented from the vertices $4$, $5$, and $6$ towards the vertices $1$, $2$, and $3$ as they were in Figure~\ref{Z3Z3semidirectZ2}.  The automorphisms $f$ and $g$ are induced by homeomorphisms of $(S^3,\Gamma_3)$.  However, now there is no homeomorphism of $(S^3, \Gamma_3)$ which interchanges the sets $\{1,2,3\}$ and $\{4,5,6\}$.  Thus $\mathrm{TSG}_+(\Gamma_3)=\langle f,g\rangle\cong \mathbb{Z}_3\times \mathbb{Z}_3$.
\end{proof}
\medskip

The following Corollary is a summary of our results for $M_3$.

\begin{corollary} All subgroups of $S_6$ containing no element of order $4$ or $5$ and no transposition are positively realized by $M_3$.\end{corollary}

\medskip


\section{Embeddings of $M_n$ for $n>3$.}

Simon~ \cite{Si} proved that for $n\geq 4$, every automorphism of $M_n$, leaves the $2n$-gon setwise invariant. Thus, for $n\geq 4$, $\mathrm{Aut}(M_n)\cong D_{2n}$, the dihedral group of order $4n$.  Note that every nontrivial subgroup of $D_{2n}$ is isomorphic to $D_k$ or $\mathbb Z_k$,  where $k$ divides $2n$.  We will now prove that each of these dihedral and cyclic groups is positively realizable for $M_n$.

\begin{theorem} For all $n\geq 4$, every subgroup of $D_{2n}$ is positively realizable for $M_n$.
\end{theorem}

\begin{proof}  Consider the embedding of $M_4$ shown in Figure~\ref{m4_1}.   We can obtain a similar embedding for any $M_5$ by adding a vertex between vertices $1$ and $8$ and a vertex between vertices $4$ and $5$, and a new rung that goes above $\overline{15}$.  By repeating this process we obtain a similar embedding $\Gamma$ of $M_n$ for any $n\geq 4$.

\begin{figure}[htbp]
\begin{center}
\includegraphics[scale=.5]{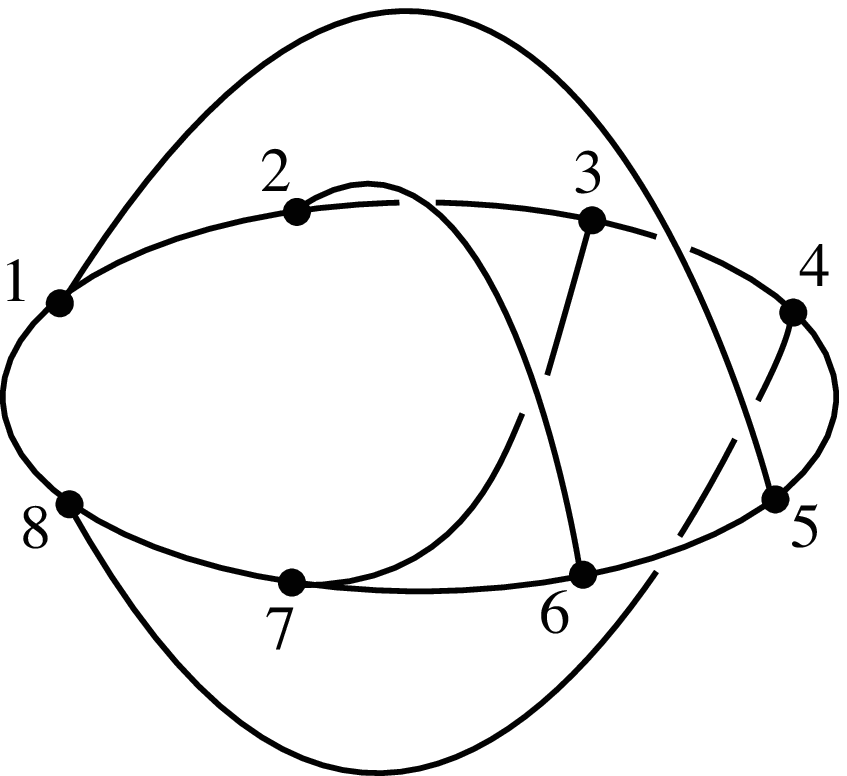}
\caption{$\mathrm{TSG}_+(\Gamma)= D_{2n}$}
\label{m4_1}
\end{center}
\end{figure}

Let $\alpha$ be a homeomorphism of $(S^3, \Gamma)$ which turns $\Gamma$ over.  Note that $\alpha$ pointwise fixes an edge precisely when $n$ is odd.  Let $\beta$ be a glide rotation of $S^3$ which rotates the $2n$-gon by $\frac{\pi}{n}$ and rotates meridionally around the $2n$-gon by $\frac{2\pi}{n}$ taking each rung to the next lower rung.  Then $\beta$ takes $\Gamma$ to itself inducing an order $2n$ automorphism on $\Gamma$.  It follows that $\mathrm{TSG}_+(\Gamma)= D_{2n}$.

Next we let $\Omega_1$ be obtained from $\Gamma$ by adding the same non-invertible knot to every edge of the $2n$-gon.  Let $\beta_1$ be the homeomorphism of $(S^3, \Omega_1)$ which rotates the $2n$-gon by $\frac{\pi}{n}$, rotates meridionally around the $2n$-gon by $\frac{2\pi}{n}$, and then isotopes the knots back into position.  Because of the non-invertible knots, the $2n$-gon cannot be turned over.  Thus, $\mathrm{TSG}_+(\Omega_1)\cong \mathbb{Z}_{2n}$.

In general, let $k$ be a nontrivial divisor of $2n$, and let $m=2n/k$.  Let $e_1$, \dots, $e_{2n}$ denote consecutive edges that make up the $2n$-gon.  Starting with the embedding $\Gamma$ add equivalent non-invertible knots to each of the edges $e_1$, $e_{1+m}$, \dots, $e_{2n-m}$ to obtain an embedding $\Omega_m$ of $M_n$.  Let $\beta_m$ be the homeomorphism of $(S^3, \Omega_m)$ which rotates the $2n$-gon by $\frac{m\pi}{n}$, rotates meridionally around the $2n$-gon by $\frac{2m\pi}{n}$, and then isotopes the knots back into position. Then $\beta_m$ will take $\Omega_m$ to itself and induce an automorphism of $M_n$ of order $k$.  Now, because the knots are non-invertible, no homeomorphism of $(S^3, \Omega_m)$ will turn the $2n$-gon over.  Also because of the knots, any homeomorphism of $(S^3, \Omega_m)$ induces a rotation of the $2n$-gon whose order is a factor of $k$.  Thus, $\mathrm{TSG}_+(\Omega_m)\cong \mathbb{Z}_{k}$.

Finally, let $\Gamma_m$ be obtained from the embedding $\Omega_m$ by replacing the non-invertible knots in $\Omega_m$ with equivalent invertible knots.  Then there is a homeomorphism of $(S^3, \Gamma_m)$ inducing an automorphism of $M_n$ of order $k$, and there is a homeomorphism of $(S^3,\Gamma_m)$ which turns the $2n$-gon over.  Thus $\mathrm{TSG}_+(\Omega_m)\cong D_{k}$.\end{proof}


{}

\end{document}